\documentclass[11pt]{article}

\usepackage[T1]{fontenc}
\usepackage[utf8]{inputenc}
\usepackage{lmodern}
\usepackage[margin=1in]{geometry}
\usepackage{amsmath,amssymb,amsthm,mathtools}
\usepackage{microtype}
\usepackage{booktabs,array}
\usepackage[numbers,sort&compress]{natbib}
\usepackage[hidelinks]{hyperref}
\hypersetup{pdftitle={Spectral Gap and Cost Comparisons for Component-Wise Markov Chains},pdfauthor={Youngwoo Kwon}}

\allowdisplaybreaks[1]
\theoremstyle{plain}
\newtheorem{theorem}{Theorem}[section]
\newtheorem{lemma}[theorem]{Lemma}
\newtheorem{proposition}[theorem]{Proposition}
\newtheorem{corollary}[theorem]{Corollary}
\theoremstyle{definition}
\newtheorem{remark}[theorem]{Remark}
\newtheorem{assm}{Assumption}

\newcommand{\RCW}{\mathrm{RCW}}
\newcommand{\DCW}{\mathrm{DCW}}
\newcommand{\RSG}{\mathrm{RSG}}
\newcommand{\DSG}{\mathrm{DSG}}
\newcommand{\X}{\mathsf X}
\newcommand{\R}{\mathbb R}
\newcommand{\ip}[2]{\langle #1,#2\rangle_\Pi}
\newcommand{\pnorm}[1]{\lVert #1\rVert_\Pi}
\newcommand{\opnorm}[1]{\lVert #1\rVert}

\title{Cost Comparisons for Random and Deterministic Scans in Component-Wise Markov Chains}
\author{Youngwoo Kwon\\
\small School of Statistics, University of Minnesota}
\date{\today}

\begin{document}
\maketitle

\begin{abstract}
Gibbs samplers, and more generally component-wise Markov chain Monte Carlo algorithms such as Metropolis-within-Gibbs, can be implemented using either random-scan or deterministic-scan updates. How much convergence can depend on this choice of scanning rule has been a longstanding question. We study this problem through $L^2$ spectral gaps, measuring computational cost in units of component updates.

For a $d$-block Gibbs sampler, the cost of random scan is at most twice that of any fixed deterministic scan, while the reverse cost ratio is at most of order $d^2$. We extend these comparisons to general reversible component-wise updates under the global block-wise contraction condition. If $K_j$ denotes the update of block $j$ and $P_j$ its Gibbs counterpart, and $\|K_j-P_j\|\leq \lambda_0<1$ for $j=1,\ldots,d,$. then the cost of random scan is at most $2/(1-\lambda_0)$ times that of deterministic scan, while the reverse cost ratio is of order at most $d^2/(1-\lambda_0)$. Examples show that the cost bounds for random scan relative to deterministic scan are asymptotically sharp and that the joint dependence on $d$ and $(1-\lambda_0)^{-1}$ in the reverse comparison cannot be improved uniformly.

This work was assisted by generative AI, including for formal verification of mathematical results in Lean. The human author reviewed and verified the mathematical content and takes full responsibility for the results.
\end{abstract}

\section{Introduction}\label{sec:introduction}

Suppose that we want to generate a sample from a multivariate target distribution $\Pi$ using a Markov chain Monte Carlo (MCMC) method \citep{broo:etal:2011}. When direct sampling is difficult, we can construct a Markov chain with invariant distribution $\Pi$ by dividing the coordinates into blocks and updating one block at a time while holding the others fixed. These chains are called component-wise Markov chains. They include Gibbs samplers, which update the selected block by drawing from its full conditional distribution given the remaining blocks \citep{Geman1984,GelfandSmith1990}, and Metropolis-within-Gibbs algorithms, which replace one or more Gibbs updates with conditional Metropolis--Hastings updates \citep{metropolis1953equation,hastings1970monte,JonesRobertsRosenthal2014}. Many MCMC methods, including posterior simulation for Bayesian hierarchical models and data-augmentation algorithms, use such component-wise updates \citep{GelfandSmith1990,TannerWong1987,ascolani2024scalability}.

For a component-wise algorithm with $d$ blocks, we need to choose both the update for each block and how the block updates are performed. The latter choice is called the scanning rule \citep{johnson2013component}. Among the various scanning rules, two common choices are random scan, which selects a block uniformly and independently at each step, and deterministic scan, which updates all blocks in a fixed order during each sweep. Since a random-scan chain formed from reversible component-wise updates is itself reversible and can therefore be analyzed using standard self-adjoint theory, whereas the corresponding deterministic-scan chain is generally nonreversible and more difficult to analyze, determining which scanning rule converges faster has long been a natural and challenging question \citep{QinJones2023}.

We address this question in terms of the $L^2(\Pi)$ spectral gap. Let $K$ be a $\Pi$-invariant Markov operator on $L^2(\Pi)$ and let $\|\cdot\|$ denote the operator norm on this space. With $\Pi$ as an averaging operator $f\mapsto\Pi(f)$, define the $L^2(\Pi)$ spectral gap of $K$ by  $\eta_K=1-\|K-\Pi\|.$
If $K$ has a positive spectral gap $\eta_K>0$, then the chain is geometrically ergodic in $L^2(\Pi)$. Geometric ergodicity also yields central limit theorems under moment conditions \citep{chan:geye:1994,jones2004markov} and permits asymptotically valid Monte Carlo standard errors under suitable conditions for variance estimation \citep{jone:hara:caff:neat:2006,vats:etal:moa:2019}. In general, a larger spectral gap gives a faster geometric bound on convergence.

Many earlier studies comparing the convergence of random and deterministic scans have mainly focused on Gibbs samplers. These results show that positivity of the gap for one scanning rule implies positivity for random scan and every deterministic ordering \citep{ChlebickaLatuszynskiMiasojedow2025,GaitondeMossel2024}. Extensions to blocked and collapsed Gibbs samplers identify conditions under which gap positivity is inherited when the updates are changed \citep{mak2026extensionssolidarityprinciplespectral}. Quantitative comparisons further bound the deterministic-scan gap below by the random-scan gap divided by a factor of order $d$. Bounds for the random-scan gap in terms of the deterministic-scan gap are quadratic in the latter and incur additional polynomial and logarithmic losses, with the stronger bound requiring a common gap lower bound over all deterministic orderings \citep{GaitondeMossel2024}.

More explicit comparisons are available in particular settings. For Gaussian targets, convergence rates describe how the comparison depends on the covariance structure, blocking and ordering \citep{AmitGrenander1991,RobertsSahu1997}. In the two-component setting, an exact relationship between the asymptotic $L^2$ convergence rates of Gibbs samplers shows that deterministic scan remains at least as fast after accounting for fixed computing times of the component updates \citep{QinJones2023}. For two reversible updates, deterministic alternation also gives no larger asymptotic variance than their uniform random mixture under covariance summability, when observations are recorded after each update \citep{Andrieu2016}. Examples based on mixing times in total variation favor either scan, exhibit polynomial differences, and show that random scan can lie between the fastest and slowest deterministic orderings \citep{roberts2015surprising,he2016scan}.

When one or more Gibbs updates are replaced by general component-wise updates, such as Metropolis--Hastings steps, the random-scan and deterministic-scan chains have been analyzed by comparison with the corresponding Gibbs samplers. In the two-component setting, geometric-ergodicity implications connect Gibbs samplers and conditional Metropolis--Hastings variants \citep{QinJones2023}. Under suitable conditions on the local proposals, uniform ergodicity of a Gibbs sampler implies uniform ergodicity of the corresponding conditional Metropolis--Hastings sampler \citep{JonesRobertsRosenthal2014}.  Quantitative comparisons for reversible hybrid random-scan chains relate their absolute spectral gaps to the Gibbs gap through uniform conditional absolute spectral-gap bounds \citep{QinJuWang2025}. Conditional conductance bounds provide another comparison, including versions localized to sets of high target probability and applications to high-dimensional Bayesian models \citep{ascolani2024scalability}. For two-component deterministic-scan Metropolis-within-Gibbs chains, weak Poincar\'e inequalities give explicit subgeometric convergence bounds in terms of the Gibbs sampler and the conditional update kernels \citep{gao2026weak}.

Notably, there exists an earlier result established spectral-gap solidarity for general reversible component-wise updates under a global block-wise contraction condition, together with polynomial quantitative comparisons between random and deterministic scans \citep{kwon2026solidarity}. In particular, the earlier lower bounds for the random-scan gap are quadratic in the deterministic gap and therefore do not give a cost comparison that is uniform in the size of the deterministic gap.

In this paper, we follow this framework to substantially sharpen these comparisons.
We show that random scan can be at most a constant factor more costly than its deterministic-scan counterpart, whereas deterministic scan can be more costly than random scan by a factor of order $d^2$, where cost is measured on the reciprocal-gap scale in units of component updates. The constants depend on the local contraction parameter. For Gibbs updates, the resulting bounds are consequences of existing projection inequalities, as discussed below.

For a precise statement, write $[d]=\{1,\ldots,d\}$, $x=(x_1,\ldots,x_d)$, and $x_{-j}=(x_i)_{i\ne j}$. Let $P_j$ denote the Gibbs update of block $j$, which is conditional expectation given $x_{-j}$ and hence an orthogonal projection on $L^2(\Pi)$. Let $K_j$ be a local update of the same block that is reversible with respect to its full conditional distribution. The random-scan component-wise (RCW) and deterministic-scan component-wise (DCW) kernels, together with their random-scan Gibbs (RSG) and deterministic-scan Gibbs (DSG) counterparts, are
\begin{align}
P_{\RCW}&=\frac1d\sum_{j=1}^dK_j,
& P_{\DCW}&=K_d\cdots K_1,
\label{eq:intro-scans}\\
P_{\RSG}&=\frac1d\sum_{j=1}^dP_j,
& P_{\DSG}&=P_d\cdots P_1.
\label{eq:intro-gibbs-scans}
\end{align}
We abbreviate their $L^2(\Pi)$ gaps as $\eta_{\RCW}$, $\eta_{\DCW}$, $\eta_{\RSG}$, and $\eta_{\DSG}$.

\begin{assm}[Global block-wise contraction]\label{assm:gbc}
There exists a common $\lambda_0\in[0,1)$ such that
\begin{equation}\label{eq:gbc}
\|K_j-P_j\|\leq\lambda_0<1,
\qquad j\in[d].
\end{equation}
\end{assm}
This assumption was introduced to compare general component-wise scans through their Gibbs counterparts \citep{kwon2026solidarity}. Since $K_j$ fixes functions of $x_{-j}$, it imposes a uniform contraction on the subspace where $P_jf=0$, including control of negative local spectral values. Note that Gibbs updates simply satisfy Assumption~\ref{assm:gbc} with $\lambda_0=0$. For conditional independence Metropolis--Hastings updates, a uniform proposal-domination bound is sufficient\citep{MengersenTweedie1996}, as described in Section~\ref{sec:mwg}.

Under Assumption~\ref{assm:gbc}, a previous comparison gives 
\begin{align*}
    \eta_{\RCW}\geq c_{\RCW}\frac{(1-\lambda_0)}{(1+\lambda_0)d^4\log^2d } \, \eta_{\DCW}^2
\end{align*}
from one deterministic ordering when $d\geq8$, where $c_{\RCW}>0$ is universal. If every deterministic ordering has gap at least $\eta_{\DCW}>0$, the bound improves to 
\begin{align*}
    \eta_{\RCW}\geq c_{\RCW}\frac{(1-\lambda_0)}{(1+\lambda_0)d^2\log^2d } \, \eta_{\DCW}^2
\end{align*}
for $d\geq2$ \citep[Theorem~1.3]{kwon2026solidarity}. The following theorem is our first result, replacing these quadratic comparisons by a linear bound with dependence $d^{-1}$.

\begin{theorem}\label{thm:comparison}
Suppose Assumption~\ref{assm:gbc} holds. For any fixed deterministic ordering with $\eta_{\DCW}>0$,
\begin{equation}\label{eq:main-bound}
\eta_{\RCW}\geq
\frac{1-\lambda_0}{d}\frac{\eta_{\DCW}}{2-\eta_{\DCW}}
\geq\frac{1-\lambda_0}{2d}\eta_{\DCW}.
\end{equation}
\end{theorem}

To account for the different costs of the two scan rules, we count individual component updates \citep{RobertsSahu1997,roberts2015surprising,he2016scan}. A random-scan step uses one update, while a deterministic sweep uses $d$ updates. On the inverse spectral gap scale, the corresponding costs are $\eta_{\RCW}^{-1}$ and $d\eta_{\DCW}^{-1}$, with analogous quantities $\eta_{\RSG}^{-1}$ and $d\eta_{\DSG}^{-1}$ for Gibbs samplers. Theorem~\ref{thm:comparison} gives
\begin{equation}\label{eq:intro-cost}
\frac{\eta_{\RCW}^{-1}}{d\eta_{\DCW}^{-1}}
=\frac{\eta_{\DCW}}{d\eta_{\RCW}}
\leq\frac{2-\eta_{\DCW}}{1-\lambda_0}
\leq\frac2{1-\lambda_0}.
\end{equation}
The factor $d^{-1}$ in~\eqref{eq:main-bound} accounts for the work in a deterministic sweep, while linear dependence on $\eta_{\DCW}$ removes the deterministic gap from the uniform cost bound. Consequently, for fixed $\lambda_0$, the bound can be independent of both the number of blocks and the size of the deterministic gap.

Taking $K_j=P_j$ and $\lambda_0=0$, we obtain the corresponding results for Gibbs samplers
\begin{align}
\eta_{\RSG}&\geq\frac{\eta_{\DSG}}{d(2-\eta_{\DSG})}
\geq\frac{\eta_{\DSG}}{2d},
\label{eq:intro-gibbs}\\
\frac{\eta_{\RSG}^{-1}}{d\eta_{\DSG}^{-1}}
&\leq2-\eta_{\DSG}\leq2,
\qquad \eta_{\DSG}>0.
\label{eq:intro-gibbs-cost}
\end{align}
Thus random scan costs at most twice as much as any fixed deterministic ordering with positive gap. Factor-two comparisons also arise in slow-convergence limits of Gaussian calculations based on asymptotic rates, under additional covariance and ordering assumptions \citep{RobertsSahu1997}. Here the factor $2$ in~\eqref{eq:intro-gibbs-cost} cannot be reduced uniformly over Gibbs targets, as shown by Proposition~\ref{prop:gaussian} in Section~\ref{sec:gibbs}.

Surprisingly, closely related bounds arise in quantum information theory, where products of orthogonal projections describe sequential projective measurements. In this setting, quantum union bounds control the loss under a sequence of measurements by the sum of the individual losses. Since Gibbs updates are orthogonal projections, one such bound gives $\eta_{\RSG}\geq\eta_{\DSG}(2-\eta_{\DSG})/(4d)$ \citep{Gao2015QuantumUnion}. A sharper form, valid on separable Hilbert spaces, gives precisely~\eqref{eq:intro-gibbs} \citep[Theorem~1.3 and Remark~1.6]{ODonnellVenkateswaran2022}. Another result for products of projections bounds the deterministic-scan gap in terms of the random-scan gap, giving
\begin{equation}\label{eq:intro-detectability}
\eta_{\DSG}\geq
1-\left(1+\frac{d\eta_{\RSG}}{(d-1)^2}\right)^{-1/2}
\end{equation}
for every fixed ordering \citep[Lemma~1 and Corollary~1]{AnshuAradVidick2016}. Thus both directions of the Gibbs comparison follow from existing projection results. Section~\ref{sec:gibbs} records these specializations to interpret their costs and examine sharpness through examples.

For the general reversible local updates, our example in Section~\ref{sec:mwg} shows that the cost bound $2/(1-\lambda_0)$ is asymptotically sharp for each fixed $\lambda_0\in(0,1)$, so its dependence on $\lambda_0$ cannot be removed uniformly. When the local kernels are positive semidefinite, however, Corollary~\ref{cor:positive} gives $\eta_{\RCW}\geq\eta_{\DCW}/\{d(2-\eta_{\DCW})\}$ without requiring Assumption~\ref{assm:gbc}. In this case, random scan again costs at most twice as much as any fixed deterministic ordering with positive gap.

For a lower bound on the deterministic-scan gap in terms of the random-scan gap, the earlier comparison under Assumption~\ref{assm:gbc} is
\begin{align*}
\eta_{\DCW}\geq\frac{(1-\lambda_0)^2}{8(d+1)(1+\lambda_0)}  \,  \eta_{\RCW}
\end{align*}
 \citep[Theorem~1.1]{kwon2026solidarity}. This bound is already linear in $\eta_{\RCW}$, but its dependence on $1-\lambda_0$ is quadratic. Theorem~\ref{thm:forward-comparison} replaces the latter dependence by a linear one.
\begin{theorem}\label{thm:forward-comparison}
Suppose Assumption~\ref{assm:gbc} holds. For any fixed deterministic ordering,
\begin{equation}\label{eq:intro-forward}
\eta_{\DCW}\geq
\frac{1-\lambda_0}
{2\{2/d+(1+\lambda_0)^2(d-1)\}}\eta_{\RCW}
\geq\frac{1-\lambda_0}{8d}\eta_{\RCW}.
\end{equation}
\end{theorem}
Converting this gap bound to component-update costs gives, whenever $\eta_{\RCW}>0$,
\begin{equation}\label{eq:deterministic-cost}
\frac{d\eta_{\DCW}^{-1}}{\eta_{\RCW}^{-1}}
=\frac{d\eta_{\RCW}}{\eta_{\DCW}}
\leq\frac{4+2(1+\lambda_0)^2d(d-1)}{1-\lambda_0}
\leq\frac{8d^2}{1-\lambda_0}.
\end{equation}
Thus, for fixed $\lambda_0$, the bound allows deterministic scan to be quadratically more costly in the number of blocks, whereas the random-scan cost bound in~\eqref{eq:intro-cost} is independent of that number. Sections~\ref{sec:gibbs} and~\ref{sec:mwg} compute the norm-gap cost asymptotics for the previously studied discrete pyramid example \citep{roberts2015surprising,GaitondeMossel2024} and its Metropolis variant. These show that the orders $d^2$ and $d^2/(1-\lambda_0)$, respectively, cannot be improved uniformly, with the latter dependence sharp jointly in $d$ and $\lambda_0$.

Section~\ref{sec:setup} introduces the notation and assumptions. Section~\ref{sec:gibbs} records the Gibbs specializations and analyzes examples of their sharpness. Section~\ref{sec:mwg} treats Metropolis-within-Gibbs algorithms and the dependence on the local contraction parameter. Section~\ref{sec:general} proves the general comparisons and the results for positive semidefinite local updates. Section~\ref{sec:discussion} discusses the scope of the results. Appendix~\ref{app:gaussian} gives the Gaussian calculations. Appendix~\ref{app:pyramid} gives the Gibbs and Metropolis pyramid calculations.

This project was initiated by the human author, who proposed using squared norm loss, motivated by related quantum-algorithmic considerations, to sharpen the quantitative comparisons between random and deterministic scans established in earlier joint work \citep{kwon2026solidarity}. Generative AI assisted in developing this idea into the proof strategy, deriving examples, and drafting and revising the manuscript. The human author reviewed the literature, scrutinized and verified the mathematical arguments and computations, and takes full responsibility for the article.

\section{Notation and problem setup}\label{sec:setup}

We follow the component-wise framework and notation of the earlier joint paper \citep{kwon2026solidarity}. Let $d\geq2$ and $[d]=\{1,\ldots,d\}$. For each $j\in[d]$, let $(\X_j,\mathcal B_j)$ be a Polish space equipped with its Borel $\sigma$-algebra. Set $\X=\prod_{j=1}^d\X_j$ and $\mathcal B=\bigotimes_{j=1}^d\mathcal B_j$. For $x\in\X$, write $x=(x_1,\ldots,x_d)$ and $x_{-j}=(x_i)_{i\neq j}$.

Let $\Pi$ be a probability measure on $(\X,\mathcal B)$. The real Hilbert space $L^2(\Pi)$ has inner product $\ip{f}{g}:=\int fg\,d\Pi$ and norm $\pnorm{f}:=\ip{f}{f}^{1/2}$. For $f\in L^2(\Pi)$, write $\Pi(f):=\int f\,d\Pi$, and use the same symbol $\Pi$ for the averaging operator $(\Pi f)(x):=\Pi(f)$. Thus $\Pi$ is the orthogonal projection onto the subspace of constant functions. Let $L^2_0(\Pi):=\{f\in L^2(\Pi):\Pi(f)=0\}$ and assume that it contains a nonzero function.

Let $K$ be a Markov kernel on $(\X,\mathcal B)$ with invariant distribution $\Pi$. Its associated Markov operator on $L^2(\Pi)$ is defined by
\begin{align*}
(Kf)(x):=\int f(y)K(x,dy),\qquad f\in L^2(\Pi).
\end{align*}
For a bounded linear operator $T$ on $L^2(\Pi)$, define $\opnorm{T}:=\sup\{\pnorm{Tf}:f\in L^2(\Pi),\,\pnorm{f}=1\}$. Since $K$ is a $\Pi$-invariant Markov kernel, it is a contraction on $L^2(\Pi)$ and satisfies $K\Pi=\Pi K=\Pi$. Throughout, the $L^2(\Pi)$ spectral gap of $K$ refers to the norm gap $\eta_K:=1-\opnorm{K-\Pi}$. In particular,
\begin{align}\label{eq:centered-norm}
\opnorm{K-\Pi}=\sup_{\substack{f\in L^2_0(\Pi), \, \pnorm{f}=1}}\pnorm{Kf}.
\end{align}
A Markov kernel $K$ is reversible with respect to $\Pi$ if it satisfies the detailed balance condition
\begin{align*}
\Pi(dx)K(x,dy)=\Pi(dy)K(y,dx).
\end{align*}
Equivalently, the associated Markov operator is self-adjoint on $L^2(\Pi)$, that is, $\ip{f}{Kg}=\ip{Kf}{g}$ for all $f,g\in L^2(\Pi)$.

For each $j\in[d]$, let $\Pi_{-j}$ denote the marginal distribution of $x_{-j}$ and let $\Pi_j(\cdot\mid x_{-j})$ be a regular conditional distribution of $x_j$ given $x_{-j}$, so that
\begin{align*}
\Pi(dx_j,dx_{-j})=\Pi_j(dx_j\mid x_{-j})\Pi_{-j}(dx_{-j}).
\end{align*}
For $\Pi_{-j}$-almost every $x_{-j}$, let $K_{j,x_{-j}}$ be a Markov kernel on $(\X_j,\mathcal B_j)$ reversible with respect to $\Pi_j(\cdot\mid x_{-j})$. Assume that $(x_j,x_{-j})\mapsto K_{j,x_{-j}}(x_j,A_j)$ is measurable for every $A_j\in\mathcal B_j$. Define the component-wise kernel $K_j$ on $(\X,\mathcal B)$ by
\begin{align*}
K_j\bigl((x_j,x_{-j}),d(y_j,y_{-j})\bigr):=K_{j,x_{-j}}(x_j,dy_j)\delta_{x_{-j}}(dy_{-j}),
\end{align*}
where $\delta_{x_{-j}}$ denotes the point mass at $x_{-j}$. Then $K_j$ is reversible with respect to $\Pi$ and hence self-adjoint on $L^2(\Pi)$.

The Gibbs update of block $j$ is the kernel
\begin{align*}
P_j\bigl((x_j,x_{-j}),d(y_j,y_{-j})\bigr):=\Pi_j(dy_j\mid x_{-j})\delta_{x_{-j}}(dy_{-j}).
\end{align*}
Its associated operator is the conditional expectation
\begin{align*}
(P_jf)(x_j,x_{-j})=\int f(y_j,x_{-j})\Pi_j(dy_j\mid x_{-j}).
\end{align*}
Hence $P_j$ is the orthogonal projection onto the subspace of $x_{-j}$-measurable functions. Since $K_j$ fixes every function in this subspace, $K_jP_j=P_j$. Self-adjointness then gives
\begin{align}\label{eq:anchor}
K_jP_j=P_jK_j=P_j.
\end{align}
These identities also follow from \citet[Lemma~2.1]{kwon2026solidarity}.

The random-scan component-wise kernel and deterministic-scan component-wise kernel are defined by
\begin{align*}
P_{\RCW}:=\frac1d\sum_{j=1}^dK_j,\qquad P_{\DCW}:=K_d\cdots K_1.
\end{align*}
The corresponding random-scan Gibbs and deterministic-scan Gibbs kernels are
\begin{align*}
P_{\RSG}:=\frac1d\sum_{j=1}^dP_j,\qquad P_{\DSG}:=P_d\cdots P_1.
\end{align*}
We take $1,\ldots,d$ as the deterministic-scan ordering. Any other fixed ordering is covered by relabeling the blocks.

We measure computational cost in units of component updates. A random-scan step uses one component update, whereas a deterministic sweep uses $d$ component updates. On the reciprocal-gap scale, we therefore compare $\eta_{\RCW}^{-1}$ with $d\eta_{\DCW}^{-1}$, and $\eta_{\RSG}^{-1}$ with $d\eta_{\DSG}^{-1}$ for Gibbs samplers.

For $\lambda_0\in[0,1)$, the component-wise updates are said to satisfy Assumption~\ref{assm:gbc} with parameter $\lambda_0$ if $\opnorm{K_j-P_j}\leq\lambda_0  $, for $j\in[d].$
By~\eqref{eq:anchor}, $K_j-P_j=K_j(I-P_j)$ and vanishes on the range of $P_j$. Thus $\opnorm{K_j-P_j}$ is the operator norm of $K_j$ on the orthogonal complement of the $x_{-j}$-measurable functions, and Assumption~\ref{assm:gbc} gives a common contraction bound for these conditional mean-zero components.

\section{Gibbs samplers}\label{sec:gibbs}

For Gibbs updates, $K_j=P_j$ and Assumption~\ref{assm:gbc} holds with $\lambda_0=0$. The following corollary records the resulting specializations of Theorems~\ref{thm:comparison} and~\ref{thm:forward-comparison}. 

\begin{corollary}[Gibbs specialization]\label{cor:gibbs}
For any fixed deterministic ordering,
\begin{equation}\label{eq:gibbs-gap}
\eta_{\RSG}\geq\frac{\eta_{\DSG}}{d(2-\eta_{\DSG})}.
\end{equation}
If $\eta_{\DSG}>0$, then
\begin{equation}\label{eq:gibbs-cost}
\frac{\eta_{\RSG}^{-1}}{d\eta_{\DSG}^{-1}}
\leq2-\eta_{\DSG}\leq2.
\end{equation}
Conversely,
\begin{equation}\label{eq:gibbs-gap-2}
\eta_{\DSG}\geq
\frac{\eta_{\RSG}}{2\{2/d+d-1\}}.
\end{equation}
If $\eta_{\RSG}>0$, then
\begin{equation}\label{eq:gibbs-cost-2}
\frac{d\eta_{\DSG}^{-1}}{\eta_{\RSG}^{-1}}
\leq4+2d(d-1)\leq2d^2.
\end{equation}
\end{corollary}

\begin{remark}
Both directions for Gibbs updates are already available from projection inequalities. The first follows from a sharp quantum union bound \citep[Theorem~1.3 and Remark~1.6]{ODonnellVenkateswaran2022}, and a stronger bound in the opposite direction follows from an existing product-of-projections inequality \citep[Lemma~1 and Corollary~1]{AnshuAradVidick2016}. 
\end{remark}

\subsection{Product targets}

\begin{proposition}\label{prop:product}
Let $\Pi=\Pi_1\otimes\cdots\otimes\Pi_d$, and suppose at least one marginal admits a nonzero square-integrable function with mean zero. Then
\begin{equation}\label{eq:product}
P_{\DSG}=\Pi,
\qquad \eta_{\DSG}=1,
\qquad \eta_{\RSG}=\frac1d.
\end{equation}
\end{proposition}
\begin{proof}
The Gibbs updates integrate out independent coordinates, so $P_d\cdots P_1=\Pi$ and $\eta_{\DSG}=1$. Equation~\eqref{eq:gibbs-gap} then gives $\eta_{\RSG}\geq1/d$. Choose a nonzero centered function $f$ depending only on a coordinate $x_k$. Then $P_kf=0$ and $P_jf=f$ for $j\ne k$. Hence $P_{\RSG}f=(d-1)f/d$, which gives $\eta_{\RSG}\leq1/d$.
\end{proof}
The product target gives equality in~\eqref{eq:gibbs-gap}, since
\begin{equation*}
\frac{\eta_{\DSG}}{d(2-\eta_{\DSG})}
=\frac1{d(2-1)}=\frac1d=\eta_{\RSG}.
\end{equation*}

\subsection{Gaussian sharpness example}

Consider $\Pi_d=N(0,Q_d^{-1})$, where
\begin{equation}\label{eq:gaussian-target}
Q_d=
\begin{pmatrix}
1&-1/2&&&0\\
-1/2&1&-1/2&&\\
&\ddots&\ddots&\ddots&\\
&&-1/2&1&-1/2\\
0&&&-1/2&1
\end{pmatrix}.
\end{equation}
Its eigenvalues are $1-\cos(k\pi/(d+1))$, $k\in[d]$, so $Q_d$ is positive definite \citep{NoschesePasquiniReichel2013}. Each full conditional is normal with variance one. Its mean is $(x_{j-1}+x_{j+1})/2$, where $x_0=x_{d+1}=0$.

\begin{proposition}\label{prop:gaussian}
For the Gibbs sampler targeting~\eqref{eq:gaussian-target},
\begin{equation}\label{eq:gaussian-sharp}
\frac{\eta_{\RSG}^{-1}}{d\eta_{\DSG}^{-1}}\longrightarrow2
\qquad\text{as }d\to\infty.
\end{equation}
\end{proposition}

The proof is given in Appendix~\ref{app:gaussian}. Thus the factor two in~\eqref{eq:gibbs-cost} cannot be reduced uniformly over targets and dimensions. This agrees with the factor-two behavior observed in earlier Gaussian scan comparisons \citep{RobertsSahu1997}.

\subsection{The discrete pyramid}\label{sec:discrete-pyramid-gibbs}

To show that the quadratic dependence on $d$ in~\eqref{eq:gibbs-cost-2} cannot be improved uniformly, consider the discrete pyramid \citep{roberts2015surprising,GaitondeMossel2024} with target distribution $\Pi$ uniform on
\begin{equation}\label{eq:pyramid-target}
\X_d=\{0,e_1,\ldots,e_d\}\subseteq\{0,1\}^d.
\end{equation}

Identify a function with its values $f=(f_0,f_1,\ldots,f_d)$ at these states. The Gibbs update of coordinate $j$ averages the entries at $0$ and $e_j$ and leaves the others unchanged. Thus
\begin{equation}\label{eq:pyramid-local}
(P_jf)_0=(P_jf)_j=\frac{f_0+f_j}{2},
\qquad (P_jf)_i=f_i\quad(i\notin\{0,j\}).
\end{equation}

\begin{proposition}\label{prop:pyramid}
For~\eqref{eq:pyramid-local},
\begin{equation}\label{eq:pyramid-gaps}
\eta_{\RSG}=\frac1{2d},
\qquad \eta_{\DSG}\sim\frac{\pi^2}{d^2}.
\end{equation}
Consequently,
\begin{equation}\label{eq:pyramid-cost}
\frac{d\eta_{\DSG}^{-1}}{\eta_{\RSG}^{-1}}
\sim\frac{d^2}{2\pi^2}.
\end{equation}
\end{proposition}

\begin{proof}
Let $a_j\in\R^{d+1}$ have entry one at index $0$, minus one at index $j$, and zero elsewhere. By~\eqref{eq:pyramid-local},
\begin{equation*}
P_j=I-\frac12a_ja_j^\top.
\end{equation*}
Consequently,
\begin{equation*}
2d(I-P_{\RSG})
=\sum_{j=1}^d a_ja_j^\top
=
\begin{pmatrix}
d&-1&\cdots&-1\\
-1&1&&0\\
\vdots&&\ddots&\\
-1&0&&1
\end{pmatrix}.
\end{equation*}

The constant vector is an eigenvector with eigenvalue zero. Every vector $f$ satisfying $f_0=0$ and $\sum_{j=1}^d f_j=0$ is an eigenvector with eigenvalue one. These vectors form a $(d-1)$-dimensional subspace. The remaining eigenvector $(d,-1,\ldots,-1)$ has eigenvalue $d+1$. Thus the eigenvalues of $2d(I-P_{\RSG})$ are
\begin{equation*}
0,\qquad
\underbrace{1,\ldots,1}_{d-1\text{ times}},
\qquad d+1.
\end{equation*}

Since $P_{\RSG}$ is an average of orthogonal projections, it is positive semidefinite. The largest eigenvalue of $P_{\RSG}$ on the mean-zero subspace is therefore $1-1/(2d)$. Therefore, $\eta_{\RSG}=1/({2d})$.

The deterministic-scan asymptotic is established in Appendix~\ref{app:pyramid}, and the cost ratio follows by substitution.
\end{proof}

Relabeling the coordinates conjugates one deterministic ordering to another by a norm-preserving map. Thus every ordering has the same gap in this example. Proposition~\ref{prop:pyramid} shows that the quadratic dependence on $d$ in~\eqref{eq:gibbs-cost-2} cannot be improved uniformly.

\section{Metropolis-within-Gibbs algorithms}\label{sec:mwg}

We next consider Metropolis-within-Gibbs algorithms, in which one or more Gibbs updates are replaced by conditional Metropolis--Hastings updates. Suppose the $j$th full conditional has density $\pi_j(\cdot\mid x_{-j})$. Given $x=(x_j,x_{-j})$, propose $y_j$ from a density $q_j(x_j,\cdot\mid x_{-j})$ and accept it with probability
\begin{equation}\label{eq:mh-acceptance}
\alpha_j(x_j,y_j\mid x_{-j})
=1\wedge
\frac{\pi_j(y_j\mid x_{-j})q_j(y_j,x_j\mid x_{-j})}
{\pi_j(x_j\mid x_{-j})q_j(x_j,y_j\mid x_{-j})}.
\end{equation}
The other coordinates remain fixed, and a rejected proposal leaves the state unchanged. The resulting kernel is reversible with respect to the full conditional \citep{hastings1970monte,JonesRobertsRosenthal2014}. Hence, whenever the resulting component-wise kernels satisfy Assumption~\ref{assm:gbc}, the general comparisons apply directly.

One convenient sufficient condition for Assumption~\ref{assm:gbc} is available for conditional independence proposals. If $q_j$ does not depend on the current value $x_j$ and, for some $\varepsilon\in(0,1]$ independent of the block and conditioning value,
\begin{equation}\label{eq:independence-domination}
q_j(y_j\mid x_{-j})
\geq\varepsilon\pi_j(y_j\mid x_{-j})
\end{equation}
for almost every $x_{-j}$ and $y_j$, then $\opnorm{K_j-P_j}\leq1-\varepsilon$.
Indeed, the domination condition implies $K_j\geq\varepsilon P_j$ \citep{MengersenTweedie1996}. For $\varepsilon<1$, write $K_j=\varepsilon P_j+(1-\varepsilon)R_j$ for a reversible conditional kernel $R_j$. Using $R_jP_j=P_jR_j=P_j$ gives $\opnorm{K_j-P_j}\leq1-\varepsilon$. If $\varepsilon=1$, then $K_j=P_j$. This sufficient condition is also discussed in the earlier component-wise framework \citep[Remark~2.5]{kwon2026solidarity}.

\subsection{A Gaussian example with sharp contraction dependence}

Use the Gaussian target in~\eqref{eq:gaussian-target}. Write $\mu_j=(x_{j-1}+x_{j+1})/2$ with $x_0=x_{d+1}=0$. Fix $\lambda\in(0,1)$ and use the conditional autoregressive proposal
\begin{equation}\label{eq:overrelaxed-update}
Y_j=\mu_j-\lambda(X_j-\mu_j)+\sqrt{1-\lambda^2}\,Z,
\qquad Z\sim N(0,1),
\end{equation}
where $Z$ is independent of the current state. This is a Gaussian overrelaxation update \citep[Section~2]{Neal1995}. Given the other coordinates and at stationarity, $(X_j-\mu_j,Y_j-\mu_j)$ is a bivariate normal vector with unit variances and correlation $-\lambda$. Its distribution is unchanged by exchanging the two entries. Hence the proposal satisfies detailed balance and its Metropolis--Hastings acceptance probability is one. For $\lambda>0$, the update is not a Gibbs draw because its conditional distribution depends on the current $X_j$.

\begin{proposition}\label{prop:gaussian-mh}
For~\eqref{eq:overrelaxed-update}, $\|K_j-P_j\|=\lambda$ for every block. For each fixed $\lambda\in(0,1)$,
\begin{equation}\label{eq:gaussian-mh-sharp}
\frac{\eta_{\RCW}^{-1}}{d\eta_{\DCW}^{-1}}
\longrightarrow\frac2{1-\lambda}
\qquad\text{as }d\to\infty.
\end{equation}
The ratio of $\eta_{\RCW}$ to the lower bound in~\eqref{eq:main-bound}, with $\lambda_0=\lambda$, tends to one.
\end{proposition}

Appendix~\ref{app:gaussian} proves this result by computing the deterministic-scan norm and evaluating random scan on a linear function. Since the local contraction parameter remains fixed as $d$ increases, both the factor two and the dependence on $1-\lambda_0$ are necessary in~\eqref{eq:intro-cost}. The same conclusion holds along sequences $\lambda=\lambda_d$ for which $d(1-\lambda_d)\to\infty$, as shown in the appendix.

\subsection{A Metropolis pyramid example}

Consider the uniform target on~\eqref{eq:pyramid-target}. Fix $\lambda\in(0,1)$. When updating coordinate $j$, propose to flip its bit with probability $(1+\lambda)/2$ and otherwise keep its current value. Accept a flip if the proposed state belongs to $\X_d$, and reject it otherwise. The proposal is symmetric. On the conditional pair $\{0,e_j\}$, the target is uniform and every flip is accepted. On every other conditional fiber, the target is degenerate and a proposed flip is rejected. The resulting conditional Metropolis--Hastings kernel is
\begin{equation}\label{eq:metropolis-pyramid-local}
K_j=(1+\lambda)P_j-\lambda I.
\end{equation}
Its nonconstant eigenvalue on the pair $\{0,e_j\}$ is $-\lambda$.

\begin{proposition}\label{prop:metropolis-pyramid}
For~\eqref{eq:metropolis-pyramid-local} and every $d\geq2$,
\begin{equation}\label{eq:metropolis-pyramid-exact}
\|K_j-P_j\|=\lambda,
\qquad
\eta_{\RCW}=\frac{1+\lambda}{2d}.
\end{equation}
Uniformly over $\lambda\in(0,1)$ as $d\to\infty$,
\begin{equation}\label{eq:metropolis-pyramid-gaps}
\eta_{\DCW}\sim
\frac{1-\lambda}{1+\lambda}\frac{\pi^2}{d^2},
\qquad
\frac{d\eta_{\DCW}^{-1}}{\eta_{\RCW}^{-1}}
\sim
\frac{(1+\lambda)^2}{2\pi^2(1-\lambda)}d^2.
\end{equation}
\end{proposition}
\begin{proof}
Since $K_j-P_j=-\lambda(I-P_j)$, the local norm equals $\lambda$. Moreover,
\begin{equation*}
P_{\RCW}=(1+\lambda)P_{\RSG}-\lambda I.
\end{equation*}
The spectrum computed in the proof of Proposition~\ref{prop:pyramid} shows that the nonconstant eigenvalues of $P_{\RCW}$ are
\begin{equation*}
1-\frac{1+\lambda}{2d}
\quad\text{with multiplicity }d-1,
\qquad
1-\frac{(1+\lambda)(d+1)}{2d}.
\end{equation*}
The first eigenvalue is nonnegative and exceeds the second. Moreover, $(1+\lambda)(d+2)\leq4d$ for $d\geq2$, so it is also at least the negative of the second. It therefore dominates the absolute value of the second. Thus the norm on the mean-zero subspace is $1-(1+\lambda)/(2d)$, which proves~\eqref{eq:metropolis-pyramid-exact}. Appendix~\ref{app:pyramid} proves the deterministic-scan asymptotic uniformly in $\lambda$. The cost ratio follows by substitution. Relabeling the leaves conjugates any deterministic ordering to the one used there, so the deterministic-scan gap is the same for every ordering.
\end{proof}

Proposition~\ref{prop:gaussian-mh} establishes the sharp dependence on $\lambda_0$ for the cost ratio of random scan to deterministic scan. Proposition~\ref{prop:metropolis-pyramid} establishes the joint orders $d^2$ and $(1-\lambda_0)^{-1}$ in~\eqref{eq:deterministic-cost}. 

\section{General component-wise algorithms}\label{sec:general}

We now allow each local update to be any reversible kernel targeting its full conditional. The proof uses self-adjointness, the identities in~\eqref{eq:anchor}, and Assumption~\ref{assm:gbc}.

\subsection{From deterministic scan to random scan}

For self-adjoint operators $A$ and $B$, write $A\succeq B$ when $\ip{f}{(A-B)f}\geq0$ for every $f\in L^2(\Pi)$. We give the short argument to make the constants explicit.

\begin{lemma}\label{lem:dissipation}
Under Assumption~\ref{assm:gbc}, for every $j\in[d]$,
\begin{equation}\label{eq:dissipation}
K_j\succeq-\lambda_0I,
\qquad I-K_j^2\succeq(1-\lambda_0)(I-K_j)\succeq0,
\end{equation}
and
\begin{equation}\label{eq:forward-operator-bounds}
(I-K_j)^2\preceq
\frac{1+\lambda_0}{1-\lambda_0}(I-K_j^2),
\qquad
I-K_j\preceq(1+\lambda_0)I.
\end{equation}
\end{lemma}
\begin{proof}
Write $f=P_jf+(I-P_j)f$. The two summands are orthogonal. Identity~\eqref{eq:anchor} shows that $K_j$ fixes the first and maps the second to a function with conditional mean zero. If $P_jf=0$, Assumption~\ref{assm:gbc} gives $\|K_jf\|_\Pi\leq\lambda_0\|f\|_\Pi$. Since $K_j$ is self-adjoint, its spectral values on such functions belong to $[-\lambda_0,\lambda_0]$. On functions fixed by $P_j$, it acts as the identity. For $t\in\{1\}\cup[-\lambda_0,\lambda_0]$,
\begin{equation*}
t\geq-\lambda_0,
\qquad
1-t^2-(1-\lambda_0)(1-t)=(1-t)(t+\lambda_0)\geq0,
\end{equation*}
and
\begin{equation*}
(1-t)^2\leq
\frac{1+\lambda_0}{1-\lambda_0}(1-t^2),
\qquad
0\leq1-t\leq1+\lambda_0.
\end{equation*}
The spectral theorem gives~\eqref{eq:dissipation} and~\eqref{eq:forward-operator-bounds}.
\end{proof}

The first inequality in~\eqref{eq:forward-operator-bounds} is the quadratic-form version of an earlier local spectral inequality \citep[Lemma~B.4]{kwon2026solidarity}. The remaining inequalities follow from the same spectral decomposition used in that work \citep[Section~3.1, Lemma~3.2 and Remark~B.2]{kwon2026solidarity}.

\begin{proof}[Proof of Theorem~\ref{thm:comparison}]
We use the sweep trajectory and telescoping squared norm losses from the earlier component-wise argument \citep[Lemma~3.1 and proof of Theorem~1.1]{kwon2026solidarity}. Here~\eqref{eq:dissipation} extends that step to the general reversible updates.

Fix $f\in L^2_0(\Pi)$ with $\pnorm{f}=1$, and define
\begin{equation*}
u_0=f,\qquad u_j=K_ju_{j-1},\quad j\in[d].
\end{equation*}
Then $u_d=P_{\DCW}f$. Set $Q=1-\pnorm{u_d}^2$. Then $\opnorm{P_{\DCW}-\Pi}\leq1-\eta_{\DCW}$ gives
\begin{equation}\label{eq:sweep-loss-lower}
Q\geq1-(1-\eta_{\DCW})^2=\eta_{\DCW}(2-\eta_{\DCW})>0.
\end{equation}
Moreover,
\begin{equation}\label{eq:telescoping}
\sum_{j=1}^d\bigl(\pnorm{u_{j-1}}^2-\pnorm{u_j}^2\bigr)=Q.
\end{equation}

Since $I-K_j$ is positive semidefinite, Cauchy--Schwarz for its quadratic form and $I-K_j^2\succeq(1-\lambda_0)(I-K_j)$ in Lemma~\ref{lem:dissipation} yield
\begin{equation}\label{eq:step-cs}
\begin{aligned}
\bigl|\ip{f}{u_{j-1}-u_j}\bigr|
&=\bigl|\ip{f}{(I-K_j)u_{j-1}}\bigr|\\
&\leq\ip{f}{(I-K_j)f}^{1/2}\ip{u_{j-1}}{(I-K_j)u_{j-1}}^{1/2}\\
&\leq\frac{\ip{f}{(I-K_j)f}^{1/2}}{\sqrt{1-\lambda_0}}
\bigl(\pnorm{u_{j-1}}^2-\pnorm{u_j}^2\bigr)^{1/2}.
\end{aligned}
\end{equation}
Since $u_0=f$ and $u_d=P_{\DCW}f$, summing and applying Cauchy--Schwarz gives
\begin{equation}\label{eq:sweep-cs}
\begin{aligned}
1-\ip{f}{P_{\DCW}f}
&=\sum_{j=1}^d\ip{f}{u_{j-1}-u_j}\\
&\leq\sum_{j=1}^d\bigl|\ip{f}{u_{j-1}-u_j}\bigr|\\
&\leq\left\{
\frac{1}{1-\lambda_0}
\sum_{j=1}^d\ip{f}{(I-K_j)f}
\sum_{j=1}^d\bigl(\pnorm{u_{j-1}}^2-\pnorm{u_j}^2\bigr)
\right\}^{1/2}\\
&=\left\{
\frac{d}{1-\lambda_0}\ip{f}{(I-P_{\RCW})f}Q
\right\}^{1/2}.
\end{aligned}
\end{equation}
On the other hand,
\begin{equation*}
\ip{f}{P_{\DCW}f}\leq\pnorm{P_{\DCW}f}=\sqrt{1-Q}.
\end{equation*}
The left side of~\eqref{eq:sweep-cs} is therefore positive. Squaring both sides and dividing by $Q$ gives
\begin{equation}\label{eq:form-bound}
\begin{aligned}
\ip{f}{(I-P_{\RCW})f}
&\geq\frac{1-\lambda_0}{d}\frac{(1-\sqrt{1-Q})^2}{Q}\\
&=\frac{1-\lambda_0}{d}\frac{1-\sqrt{1-Q}}{1+\sqrt{1-Q}}\\
&\geq\frac{1-\lambda_0}{d}\frac{\eta_{\DCW}}{2-\eta_{\DCW}}.
\end{aligned}
\end{equation}
The last step follows from~\eqref{eq:sweep-loss-lower} and the fact that $(1-\sqrt{1-x})/(1+\sqrt{1-x})$ is increasing on $[0,1]$.

It remains to control the lower end of the random-scan spectrum. Set
\begin{equation*}
c=\frac{1-\lambda_0}{d}\frac{\eta_{\DCW}}{2-\eta_{\DCW}}.
\end{equation*}
Lemma~\ref{lem:dissipation} gives $P_{\RCW}\succeq-\lambda_0I$, whereas \eqref{eq:form-bound} gives $P_{\RCW}\preceq(1-c)I$ on $L^2_0(\Pi)$. Thus the spectrum relevant to~\eqref{eq:centered-norm} belongs to $[-\lambda_0,1-c]$. Since $0<c\leq(1-\lambda_0)/d\leq1-\lambda_0$, we have $\lambda_0\leq1-c$. Self-adjointness now implies
\begin{equation*}
\opnorm{P_{\RCW}-\Pi}\leq\max\{\lambda_0,1-c\}=1-c.
\end{equation*}
This proves the first inequality in~\eqref{eq:main-bound}. The second follows from $2-\eta_{\DCW}\leq2$.
\end{proof}

\subsection{From random scan to deterministic scan}

\begin{proof}[Proof of Theorem~\ref{thm:forward-comparison}]
This proof refines the earlier random-to-deterministic comparison argument \citep[Theorem~1.1]{kwon2026solidarity}.
Put $\varepsilon=1-\lambda_0$ and $s=1+\lambda_0$. Fix $f\in L^2_0(\Pi)$ with $\pnorm{f}=1$, and define
\begin{equation*}
u_0=f,
\qquad
u_j=K_ju_{j-1},
\qquad
q_j=\pnorm{u_{j-1}}^2-\pnorm{u_j}^2.
\end{equation*}
Since
\begin{equation*}
q_j=\ip{u_{j-1}}{(I-K_j^2)u_{j-1}}\geq0,
\end{equation*}
the losses telescope to
\begin{equation}\label{eq:forward-loss}
Q:=\sum_{j=1}^dq_j
=1-\pnorm{P_{\DCW}f}^2.
\end{equation}
By~\eqref{eq:forward-operator-bounds},
\begin{equation}\label{eq:forward-increment}
\pnorm{(I-K_j)u_{j-1}}^2
\leq\frac{s}{\varepsilon}q_j.
\end{equation}
For each $j$, Cauchy--Schwarz and~\eqref{eq:forward-increment} give
\begin{equation}\label{eq:forward-displacement}
\begin{aligned}
\pnorm{f-u_{j-1}}^2
&=\pnorm{\sum_{i<j}(I-K_i)u_{i-1}}^2\\
&\leq(j-1)\sum_{i<j}\pnorm{(I-K_i)u_{i-1}}^2\\
&\leq\frac{s(j-1)}{\varepsilon}Q.
\end{aligned}
\end{equation}
The right side is zero when $j=1$.

For a positive semidefinite operator $A$, Cauchy--Schwarz for its quadratic form gives
\begin{equation*}
\ip{v+w}{A(v+w)}
\leq2\ip{v}{Av}+2\ip{w}{Aw}.
\end{equation*}
Apply this with $A=I-K_j$, $v=u_{j-1}$, and $w=f-u_{j-1}$. Equations~\eqref{eq:dissipation}, \eqref{eq:forward-operator-bounds}, and~\eqref{eq:forward-displacement} yield
\begin{equation}\label{eq:forward-form-step}
\ip{f}{(I-K_j)f}
\leq\frac{2}{\varepsilon}q_j
+\frac{2s^2(j-1)}{\varepsilon}Q.
\end{equation}
Averaging over $j$ and using $\sum_{j=1}^d(j-1)=d(d-1)/2$ gives
\begin{equation}\label{eq:forward-form-comparison}
\ip{f}{(I-P_{\RCW})f}
\leq\frac1{\varepsilon}
\left\{\frac2d+s^2(d-1)\right\}Q.
\end{equation}

The operator $P_{\RCW}$ is self-adjoint, but it need not be positive semidefinite. Nevertheless, since $f$ is centered,
\begin{equation*}
\ip{f}{P_{\RCW}f}
=\ip{f}{(P_{\RCW}-\Pi)f}
\leq\opnorm{P_{\RCW}-\Pi},
\end{equation*}
and hence
\begin{equation*}
\eta_{\RCW}\leq\ip{f}{(I-P_{\RCW})f}.
\end{equation*}
Combining this inequality with~\eqref{eq:forward-form-comparison} and taking the infimum over centered unit vectors gives
\begin{equation}\label{eq:forward-squared-loss}
\begin{aligned}
1-\opnorm{P_{\DCW}-\Pi}^2
&=\inf_{\substack{\Pi(f)=0\\\pnorm{f}=1}}
\{1-\pnorm{P_{\DCW}f}^2\}\\
&\geq
\frac{\varepsilon}{2/d+s^2(d-1)}\eta_{\RCW}.
\end{aligned}
\end{equation}
The left side of~\eqref{eq:forward-squared-loss} is the squared-norm loss
\begin{equation*}
1-(1-\eta_{\DCW})^2
=\eta_{\DCW}(2-\eta_{\DCW}),
\end{equation*}
not the spectral gap itself. Since this quantity is at most $2\eta_{\DCW}$, the  first inequality in Theorem~\ref{thm:forward-comparison} follows. The second follows from $s\leq2$ and
\begin{equation*}
2\{2/d+s^2(d-1)\}\leq8d,
\qquad d\geq2.
\end{equation*}
Rearranging the first inequality gives the refined cost bound in~\eqref{eq:deterministic-cost}. Its final inequality follows from $s\leq2$. Relabeling the blocks proves the result for any fixed deterministic ordering.
\end{proof}

\subsection{Positive semidefinite local updates}

The dependence on $\lambda_0$ can be removed when every $K_j$ is positive semidefinite. In this case the spectral theorem gives
\begin{equation*}
K_j\succeq0,
\qquad I-K_j^2\succeq I-K_j.
\end{equation*}
These are precisely the two inequalities used in the proof of Theorem~\ref{thm:comparison} with coefficient $1-\lambda_0$ replaced by one. This gives the following corollary, without requiring Assumption~\ref{assm:gbc}.

\begin{corollary}\label{cor:positive}
Suppose each $K_j$ is a reversible positive semidefinite component-wise Markov operator. For any fixed deterministic ordering with $\eta_{\DCW}>0$,
\begin{equation}\label{eq:positive-bound}
\eta_{\RCW}\geq\frac{\eta_{\DCW}}{d(2-\eta_{\DCW})},
\qquad
\frac{\eta_{\RCW}^{-1}}{d\eta_{\DCW}^{-1}}\leq2.
\end{equation}
\end{corollary}

For example, local data augmentation kernels are positive semidefinite. Conditional on $x_{-j}$, let a joint law of $(X_j,Z)$ have the required full conditional as its $X_j$ marginal. Draw $Z$ given $X_j$, followed by a new $X_j$ given $Z$. At stationarity, conditional independence of the old and new $X_j$ given $Z$ gives
\begin{equation*}
\ip{f}{K_jf}
=\mathbb E\left[\mathbb E\{f(X)\mid Z,X_{-j}\}^2\right]\geq0.
\end{equation*}
Likewise, if $L_j$ is a reversible conditional update, using two successive steps gives $K_j=L_j^2$ and $\ip{f}{K_jf}=\|L_jf\|_\Pi^2\geq0$.

\section{Discussion}\label{sec:discussion}

The comparison results extend the bounds available for Gibbs projections to general reversible component-wise updates and sharpen the quantitative comparisons in the earlier joint paper \citep{kwon2026solidarity}. In both directions, the lower gap bound is linear in the gap of the other scanning rule and has dimension dependence $d^{-1}$.

For the cost ratio of random scan to deterministic scan, the constant $2/(1-\lambda_0)$ is asymptotically sharp for every fixed $\lambda_0\in[0,1)$. The Gaussian overrelaxation example proves this with a fixed local contraction parameter. In the opposite direction, Theorem~\ref{thm:forward-comparison} gives a cost bound of order $d^2/(1-\lambda_0)$. The Metropolis pyramid attains these joint orders uniformly over the contraction parameter. This establishes sharpness of the parameter dependence.

Negative local spectral values explain why the contraction parameter appears in the random-to-deterministic cost bound. Positive semidefinite local updates satisfy the factor-two bound even without Assumption~\ref{assm:gbc}. In the Gaussian example attaining $2/(1-\lambda_0)$, the conditional linear mode has eigenvalue $-\lambda_0$. On this mode, the local inequality in~\eqref{eq:dissipation} is an equality.

The costs studied here are based on the $L^2$ norm gap. This gap does not determine total-variation mixing times uniformly over starting states. For example, product Gibbs targets have equal costs on the scale considered here, while random scan has a coupon-collector obstruction to updating all coordinates from a point initialization \citep[Section~2]{roberts2015surprising}.

\paragraph{Acknowledgments.}
I am grateful to my advisors, Galin Jones and Qian Qin, for their guidance and many helpful discussions throughout the development of this work.

\paragraph{Use of generative AI.}
This project was initiated by the human author, who proposed studying whether the quantitative comparison between deterministic-scan and random-scan component-wise Markov chains could be substantially sharpened from the earlier joint work \citep{kwon2026solidarity}. Motivated by squared-gap quantities arising in related quantum-algorithmic considerations, the human author proposed using squared norm loss to study the comparison. Generative AI assisted in developing this idea into the principal proof strategy, including the telescoping argument along a deterministic sweep that leads to the main spectral-gap bound. It also assisted with related derivations, examples, manuscript drafting, revisions and formal verification of the mathematical results in Lean. The human author determined the research direction, checked the relevant literature, scrutinized and revised the AI-generated arguments, verified the mathematical claims and computations, and takes full responsibility for the content and correctness of the article.

\appendix

\section{Gaussian scan comparisons}\label{app:gaussian}

We prove Propositions~\ref{prop:gaussian} and~\ref{prop:gaussian-mh} together. Throughout this appendix, let $Q=Q_d$ from~\eqref{eq:gaussian-target} and consider~\eqref{eq:overrelaxed-update} with $\lambda\in[0,1)$. At $\lambda=0$, these are Gibbs updates.

\subsection{The norm of a Gaussian transition}

We first record the Gaussian norm identity used below. Suppose $X$ and $Y$ are jointly Gaussian with the same nondegenerate centered marginal distribution. Transform each marginal to have identity covariance, then apply orthogonal transformations so that their cross-covariance is diagonal with entries $s_1,\ldots,s_d\in[0,1]$. In the resulting coordinates, the pairs $(X_i,Y_i)$ are independent and have correlations $s_i$.

For normalized Hermite polynomials, the Gaussian generating function gives
\begin{equation*}
\mathbb E[h_m(Y_i)\mid X_i]=s_i^m h_m(X_i).
\end{equation*}
The products of these polynomials form orthonormal bases for the two Gaussian marginals \citep{Janson1997Gaussian}. Therefore the singular values of conditional expectation on nonconstant products are $\prod_i s_i^{m_i}$ with $\sum_i m_i\geq1$. Their largest value is $\max_i s_i$ and is attained by a linear function. This proves that, for a stationary Gaussian transition $K$, the norm $\|K-\Pi\|$ is determined by its action on centered linear functions.

For $\Pi=N(0,Q^{-1})$, write $f_a(x)=a^\top Qx$. Then
\begin{equation*}
\|f_a\|_\Pi^2=a^\top Qa.
\end{equation*}
If $Kf_a=f_{Ba}$, the identity just proved gives
\begin{equation}\label{eq:gaussian-norm}
\|K-\Pi\|^2
=\sup_{a\ne0}\frac{(Ba)^\top Q(Ba)}{a^\top Qa}
=:\|B\|_Q^2.
\end{equation}
No statement about eigenvalues of a nonreversible $B$ is used here.

\subsection{Local contraction and the deterministic sweep}

Conditional on $x_{-j}$, the residual $x_j-\mu_j$ is standard normal. Under~\eqref{eq:overrelaxed-update}, its Hermite polynomial of degree $m$ has multiplier $(-\lambda)^m$. Expanding a square-integrable function in these polynomials, with coefficients depending on $x_{-j}$, shows that
\begin{equation}\label{eq:overrelaxed-contraction}
\|K_j-P_j\|=\sup_{m\geq1}|(-\lambda)^m|=\lambda.
\end{equation}
Equality is attained by the residual itself when $\lambda>0$.

The conditional mean of the new coordinate is $x_j-(1+\lambda)(Qx)_j$. Consequently,
\begin{equation}\label{eq:gaussian-coordinate}
K_jf_a=f_{B_ja},
\qquad B_j=I-(1+\lambda)e_je_j^\top Q.
\end{equation}
Write $Q=I+L+L^\top$, where $L$ is strictly lower triangular, and put
\begin{equation*}
M=\frac1{1+\lambda}I+L.
\end{equation*}
The successive coordinate equations give the classical Gauss--Seidel iteration matrix when $\lambda=0$ and the successive overrelaxation (SOR) matrix with relaxation parameter $\omega=1+\lambda$ in general \citep[Sections~3.1 and~3.3]{FoxParker2017}. In our notation,
\begin{equation}\label{eq:gauss-seidel}
B=B_d\cdots B_1=I-M^{-1}Q.
\end{equation}
Indeed, if $b=Ba$ and $v=a-b$, then the $j$th update uses the already updated coordinates below $j$ and the original coordinates above $j$. It follows that $v_j=(1+\lambda)\{(Qa)_j-\sum_{i<j}Q_{ji}v_i\}$, which is exactly $Mv=Qa$.

Since $M+M^\top-Q=(1-\lambda)I/(1+\lambda)$, direct expansion yields
\begin{equation}\label{eq:gs-energy}
Q-B^\top QB
=QM^{-\top}(M+M^\top-Q)M^{-1}Q
=\frac{1-\lambda}{1+\lambda}QM^{-\top}M^{-1}Q.
\end{equation}
Thus
\begin{equation}\label{eq:gs-eigenvalue}
\begin{aligned}
1-\|B\|_Q^2
&=\frac{1-\lambda}{1+\lambda}
\inf_{a\ne0}\frac{\|M^{-1}Qa\|_2^2}{a^\top Qa}\\
&=\frac{1-\lambda}{1+\lambda}
\frac1{\|Q^{-1/2}M\|_2^2}.
\end{aligned}
\end{equation}
The second equality follows by setting $z=M^{-1}Qa$. The denominator becomes $z^\top M^\top Q^{-1}Mz$. Every local update is Gaussian, so a deterministic sweep is a stationary Gaussian transition. Equation~\eqref{eq:gaussian-norm} therefore gives $\eta_{\DCW}=1-\|B\|_Q$.

\subsection{Asymptotic evaluation and sharpness}

Let $\epsilon_d=1-\cos(\pi/(d+1))$, the smallest eigenvalue of $Q_d$. Let $S$ be the lower shift matrix. Then
\begin{equation*}
Q=I-\tfrac12(S+S^\top),
\qquad
M=\frac{1-\lambda}{2(1+\lambda)}I+\tfrac12(I-S).
\end{equation*}
The identity
\begin{equation*}
(I-S)(I-S)^\top=2Q-e_1e_1^\top\preceq2Q
\end{equation*}
implies $\|Q^{-1/2}(I-S)\|_2\leq\sqrt2$. The triangle inequality and its reverse form give
\begin{equation}\label{eq:gs-sandwich}
\left|\|Q^{-1/2}M\|_2-
\frac{1-\lambda}{2(1+\lambda)\sqrt{\epsilon_d}}\right|
\leq\frac1{\sqrt2}.
\end{equation}
For fixed $\lambda<1$, equations~\eqref{eq:gs-eigenvalue} and~\eqref{eq:gs-sandwich} give
\begin{equation*}
1-\|B\|_Q^2\sim\frac{4(1+\lambda)}{1-\lambda}\epsilon_d.
\end{equation*}
Since $\|B\|_Q\to1$, this implies
\begin{equation}\label{eq:gaussian-deterministic-asymptotic}
\eta_{\DCW}
=\frac{1-\|B\|_Q^2}{1+\|B\|_Q}
\sim\frac{2(1+\lambda)}{1-\lambda}\epsilon_d.
\end{equation}

To bound the random-scan gap, choose a nonzero $a$ with $Qa=\epsilon_da$. Averaging~\eqref{eq:gaussian-coordinate} gives
\begin{equation*}
P_{\RCW}f_a
=\left(1-\frac{(1+\lambda)\epsilon_d}{d}\right)f_a.
\end{equation*}
The displayed eigenvalue is positive, so
\begin{equation}\label{eq:gaussian-random-upper}
\eta_{\RCW}\leq\frac{(1+\lambda)\epsilon_d}{d}.
\end{equation}
Together with Theorem~\ref{thm:comparison}, this gives
\begin{equation}\label{eq:gaussian-cost-sandwich}
\frac{\eta_{\DCW}}{(1+\lambda)\epsilon_d}
\leq\frac{\eta_{\RCW}^{-1}}{d\eta_{\DCW}^{-1}}
\leq\frac{2-\eta_{\DCW}}{1-\lambda}.
\end{equation}
Both outer expressions converge to $2/(1-\lambda)$ by~\eqref{eq:gaussian-deterministic-asymptotic}. This proves Proposition~\ref{prop:gaussian-mh}. Setting $\lambda=0$ proves Proposition~\ref{prop:gaussian}.

The same calculation gives
\begin{equation*}
\eta_{\RCW}\sim\frac{(1+\lambda)\epsilon_d}{d},
\qquad
\frac{\eta_{\RCW}}
{(1-\lambda)\eta_{\DCW}/\{d(2-\eta_{\DCW})\}}\longrightarrow1.
\end{equation*}
Thus the refined bound itself is asymptotically attained. Its lower bound comes from the independently proved main theorem, while the matching upper bound comes from the linear function above.

For a varying $\lambda=\lambda_d$, the same relative estimates hold whenever $\sqrt{\epsilon_d}/(1-\lambda_d)\to0$. Since $\epsilon_d\sim\pi^2/(2d^2)$, this is equivalent to $d(1-\lambda_d)\to\infty$. In this regime,
\begin{equation*}
\frac{\eta_{\RCW}^{-1}}{d\eta_{\DCW}^{-1}}
\sim\frac2{1-\lambda_d}.
\end{equation*}

\section{The pyramid sweep}\label{app:pyramid}

We prove the deterministic-scan parts of Propositions~\ref{prop:pyramid} and~\ref{prop:metropolis-pyramid} simultaneously. For the Gibbs sampler take $\lambda=0$. For the Metropolis-within-Gibbs sampler take $\lambda\in(0,1)$. Define
\begin{equation*}
a=\frac{1+\lambda}{2},\qquad
b=\frac{1-\lambda}{2},\qquad
\beta=\frac ba=\frac{1-\lambda}{1+\lambda}.
\end{equation*}
On the two entries with indices $0$ and $j$, $K_j$ acts by $\left(\begin{smallmatrix}b&a\\a&b\end{smallmatrix}\right)$. It leaves the other entries unchanged. Since $\Pi$ is uniform, for any vector of function values $v\in\R^{d+1}$,
\begin{equation*}
\pnorm{v}^2=\frac{\opnorm{v}_2^2}{d+1},\qquad\pnorm{v}=\frac{\opnorm{v}_2}{\sqrt{d+1}}.
\end{equation*}
The normalization cancels in operator-norm ratios.

\subsection{An exact loss identity}

Let $x=(x_0,\ldots,x_d)$ be the input vector of function values, and write $y=K_d\cdots K_1x$. Define the successive values at the central state by
\begin{equation}\label{eq:pyramid-recurrence}
s_0=x_0,\qquad s_j=b s_{j-1}+a x_j,\qquad j\in[d].
\end{equation}
When update $j$ occurs, the leaf value $x_j$ has not previously changed. Thus $y_0=s_d$ and $y_j=a s_{j-1}+b x_j$.

Put $\delta_j=s_j-s_{j-1}$ and $E(s)=\sum_{j=1}^d\delta_j^2$. The eigenvalues of the two-entry update are one and $-\lambda$. Its squared Euclidean norm loss is therefore $\tfrac12(1-\lambda^2)(s_{j-1}-x_j)^2$. Since $\delta_j=a(x_j-s_{j-1})$, summing the losses gives
\begin{equation}\label{eq:pyramid-energy}
\opnorm{x}_2^2-\opnorm{y}_2^2
=2\beta E(s).
\end{equation}
The input can be recovered from the path $s$ by
\begin{equation}\label{eq:pyramid-recover}
x_0=s_0,\qquad x_j=s_j+\beta(s_j-s_{j-1}),\quad j\in[d].
\end{equation}
This is an invertible linear transformation. If $D$ denotes the difference map $(Ds)_0=0$ and $(Ds)_j=s_j-s_{j-1}$ for $j\geq1$, then $x=s+\beta Ds$ and $\opnorm{Ds}_2^2=E(s)$.

\subsection{Comparison with the path Laplacian}

Let $n=d+1$, and let $C_n=I-n^{-1}\mathbf1\mathbf1^\top$ be Euclidean centering. The symmetric matrix associated with $E(s)=\sum_{j=1}^{n-1}(s_j-s_{j-1})^2$ has smallest positive eigenvalue
\begin{equation}\label{eq:path-gap}
\nu_n=2-2\cos(\pi/n)=4\sin^2(\pi/(2n)).
\end{equation}
This matrix is the path Laplacian $L_n$. Its eigenvalues are $2-2\cos(k\pi/n)$, $k=0,\ldots,n-1$, with eigenvector entries $\cos\{k\pi(i+1/2)/n\}$ for $i=0,\ldots,n-1$ \citep[Theorem~5.6.1, pp.~51--52]{Spielman2025SAGT}. In particular, the $k=0$ mode is constant and the $k=1$ mode is nonzero and orthogonal to constants. Since
\begin{equation*}
 s^\top L_n s=\sum_{j=1}^{n-1}(s_j-s_{j-1})^2=E(s),
\end{equation*}
the $k=1$ mode attains $E(s)=\nu_n\opnorm{s}_2^2$. The Courant--Fischer characterization \citep[Theorem~2.0.1, p.~21]{Spielman2025SAGT} gives $E(s)\geq\nu_n\opnorm{C_ns}_2^2$ for every $s$ by applying the Rayleigh inequality to $C_ns$, which is orthogonal to the constant mode. These are the finite-path inputs used below. Thus
\begin{equation}\label{eq:path-poincare}
\opnorm{C_ns}_2\leq\nu_n^{-1/2}E(s)^{1/2}.
\end{equation}

For a centered input $x$, equations~\eqref{eq:pyramid-recover} and~\eqref{eq:path-poincare} imply
\begin{equation*}
\begin{aligned}
\opnorm{x}_2
&=\opnorm{C_n(s+\beta Ds)}_2\\
&\leq\opnorm{C_ns}_2+\beta\opnorm{C_nDs}_2\\
&\leq(\nu_n^{-1/2}+\beta)E(s)^{1/2}.
\end{aligned}
\end{equation*}
Consequently, if $\tau_{d,\lambda}=1-\opnorm{P_{\DCW}-\Pi}^2$, the exact identity~\eqref{eq:pyramid-energy} yields the lower bound
\begin{equation}\label{eq:pyramid-loss-lower}
\tau_{d,\lambda}\geq
\frac{2\beta\nu_n}{(1+\beta\sqrt{\nu_n})^2}.
\end{equation}

For an upper bound, choose a nonzero centered first path eigenvector $s$, so $E(s)=\nu_n\opnorm{s}_2^2$. Set $x=C_n(s+\beta Ds)$ and $c=n^{-1}\mathbf1^\top(s+\beta Ds)$. Then $\widetilde s=s-c\mathbf1$ satisfies $x=\widetilde s+\beta D\widetilde s$, so it is the path associated with $x$ by~\eqref{eq:pyramid-recover}, and $E(\widetilde s)=E(s)$. The reverse triangle inequality gives
\begin{equation*}
\opnorm{x}_2\geq
\opnorm{s}_2-\beta\opnorm{C_nDs}_2
\geq(1-\beta\sqrt{\nu_n})\opnorm{s}_2.
\end{equation*}
For all sufficiently large $d$, $\beta\sqrt{\nu_n}<1$ and $x\ne0$. Using this input $x$ and its path $\widetilde s$ in~\eqref{eq:pyramid-energy} gives
\begin{equation}\label{eq:pyramid-loss-upper}
\tau_{d,\lambda}\leq
\frac{2\beta\nu_n}{(1-\beta\sqrt{\nu_n})^2}.
\end{equation}
Since $0<\beta\leq1$, we have $\beta\sqrt{\nu_n}\to0$ uniformly over $\lambda\in[0,1)$. The upper and lower bounds therefore imply, uniformly over this range,
\begin{equation*}
\tau_{d,\lambda}\sim2\beta\nu_{d+1}
\sim\frac{2\beta\pi^2}{d^2}.
\end{equation*}
Therefore
\begin{equation*}
\eta_{\DCW}
=\frac{\tau_{d,\lambda}}{1+\sqrt{1-\tau_{d,\lambda}}}
\sim\frac{\beta\pi^2}{d^2}.
\end{equation*}
Taking $\lambda=0$ proves~\eqref{eq:pyramid-gaps}. Taking $\lambda\in(0,1)$ proves~\eqref{eq:metropolis-pyramid-gaps}, including the stated uniformity. This completes the deterministic-scan calculations for the Gibbs and Metropolis-within-Gibbs examples.

\bibliographystyle{plainnat}
\begingroup
\small
\bibliography{refs}
\endgroup

\end{document}